\documentclass[a4paper,leqno,11pt]{amsart}

\usepackage{amssymb}

\theoremstyle{plain}
\newtheorem{thm}{Theorem}[section]
\newtheorem{lem}[thm]{Lemma}
\newtheorem{cor}[thm]{Corollary}

\newtheorem*{theorem*}{Theorem}

\theoremstyle{definition}
\newtheorem{defn}[thm]{Definition}
\newtheorem{rmk}[thm]{Remark}
\newtheorem{notation}[thm]{Notation}
\newtheorem{exam}[thm]{Example}
 2

\begin{document}

\title{A study on some combinatorial sets in Partial Semigroups}

\author{Arpita Ghosh}

\email{arpi.arpi16@gmail.com}

\email{arpitamath18@klyuniv.ac.in}

\address{Department of Mathematics, University of Kalyani,
Kalyani, Nadia-741235
West Bengal, India}
\begin{abstract}
In this article, we investigate the image and preimage of the important combinatorial sets such as central sets, $C$-sets, and $J_\delta$-sets  which play an important
role in the study of combinatorics under certain partial semigroup homomorphism. Using that we prove certain results which deal with the existence of $C$-set which are not central in partial semigroup framework.
\end{abstract}

\maketitle
\section{introduction}
The notion of the central subset of $\mathbb{N}$ was originally introduced by Furstenberg \cite{F81} in
terms of a topological dynamical system. Before defining central sets let us start with
original Central Sets Theorem due to Furstenberg
  
\begin{thm}(Original Central Sets Theorem)
 Let $l\in\mathbb{N}$ and for each $i\in\{1,2,\ldots,l\}$, and
let $\langle y_{i,n}\rangle_{n=1}^{\infty}$ be a sequence
in $\mathbb{Z}$. Let $C$ be a central subset of $\mathbb{N}$. Then
there exist sequences $\langle a_{n}\rangle_{n=1}^{\infty}$
in $\mathbb{N}$ and $\langle H_{n}\rangle_{n=1}^{\infty}$
in $\mathcal{P}_{f}(\mathbb{N})$ such that 
\begin{itemize}
\item[(1)] for all $n$, $\max H_{n} < \min H_{n+1}$ and
\item[(2)] for all $F\in\mathcal{P}_{f}(\mathbb{N})$ and all 
$i\in\{1,2,\ldots,l\},$ $$\sum_{n\in F}(a_{n}+\sum_{t\in H_{n}}y_{i,t})\in C$$.
\end{itemize}
\end{thm}

\begin{proof}
\cite[Proposition 8.21]{F81}.
\end{proof}

This theorem has several combinatorial consequences such as Rado's theorem which deals with the regularity of the system of integral equations. Central sets in natural numbers are known to have substantial combinatorial structure. For example, any central set contains arbitrary long arithmatic progressions, all finite sums of distinct terms of an infinite sequence.

  In \cite{BH90}, Vitaly Bergelson and Neil Hindman came with an algebraic characterization of the central subsets of natural numbers in terms of the algebra of $\beta \mathbb{N}.$ Analogously, the notion of a central set has been defined in an arbitrary discrete  semigroup $S.$ To define the central subsets in a semigroup,  we recall a brief introduction of the algebraic structure of $\beta S$ for a discrete semigroup $(S, \cdot)$.

  We take the points of $\beta S$ to be
the ultrafilters on $S$, identifying the principal ultrafilters with
the points of $S$ and thus pretending that $S\subseteq\beta S$.
Given $A\subseteq S$ let us set, $\overline{A}=\{p\in\beta S\mid A\in p\}$.
 Then the set $\{\overline{A}\mid A\subseteq S\}$ is a basis for a topology
on $\beta S$. The operation $\cdot$ on $S$ can be extended to the Stone-\v{C}ech
compactification $\beta S$ of $S$ so that $(\beta S,\cdot)$ is a compact
right topological semigroup (meaning that for any $p\in\beta S$,
the function $\rho_{p}:\beta S\rightarrow\beta S$ defined by $\rho_{p}(q)=q\cdot p$
is continuous) with $S$ contained in its topological center (meaning
that for any $x\in S$, the function $\lambda_{x}:\beta S\rightarrow\beta S$
defined by $\lambda_{x}(q)=x \cdot q$ is continuous). Given $p,q\in\beta S$
and $A\subseteq S$, $A\in p \cdot q$ if and only if $\{x\in S \mid x^{-1} \cdot A\in q\}\in p$,
where $x^{-1} \cdot A=\{y\in S \mid x\cdot y\in A\}$.

A nonempty subset $I$ of a semigroup $(T,\cdot)$ is called a left ideal
of ${T}$ if $T\cdot I\subset I$, a right ideal if $I\cdot T\subset I$,
and a two-sided ideal (or simply an ideal) if it is both a left and
a right ideal. A minimal left ideal is a left ideal that does not
contain any proper left ideal. Similarly, we can define a minimal right
ideal.

Any compact Hausdorff right topological semigroup $(T,\cdot)$ has a smallest
two sided ideal

\[
\begin{array}{rl}
K(T) = & \bigcup\{L \mid L\text{ is a minimal left ideal of }T\}\\
  = &\bigcup\{R \mid R\text{ is a minimal right ideal of }T\}
\end{array}
\]

Given a minimal left ideal $L$ and a minimal right ideal $R$, $L\cap R$
is a group, and in particular, contains an idempotent. An idempotent
in $K(T)$ is called a minimal idempotent. If $p$ and $q$ are idempotents
in $T$, we write $p\leq q$ if and only if $p\cdot q=q\cdot p=p$. An idempotent
is minimal with respect to this relation if and only if it is a member
of the smallest ideal. See \cite{HS12}  for an elementary introduction
to the algebra of $\beta S$ and for any unfamiliar
details.

Now we recall the central sets of a semigroup as

\begin{defn}
Consider a discrete semigroup $S$ and a subset $A$ of $S.$ Then $A$ is central if and only if there is an idempotent in $K(\beta S) \cap \overline{A}.$  
\end{defn}

In \cite{HS12}, the Central Sets Theorem was extended to arbitrary semigroups by allowing the choice of countably many sequences at a time.  More extended version of the Central Sets Theorem considering all sequences at one time has been established  in \cite{DHS08}. The sets which satisfy the conclusion of the above Central Sets Theorem are the objects that matter. They have several applications in the study of all non-trivial partition regular systems of homogeneous integral equations. Hindman and Strauss understand their importance and named them as $C$-sets.  Also, they found similar kind of algebraic characterizations of $C$-sets in terms of the algebra of $\beta S$  and $J$-sets, which was also introduced by them. We recall the definition of $C$-set which will be useful later.

\begin{defn}
Let $(S,+)$ be a commutative semigroup and let $A\subseteq S$ and
$\mathcal{T}=S^{\mathbb{N}}$, the set of sequences in $S$. The
set $A$ is a $C$-set if and only if there exist functions $\alpha:\mathcal{P}_{f}(\mathcal{T})\rightarrow S$
and $H:\mathcal{P}_{f}(\mathcal{T})\rightarrow\mathcal{P}_{f}(\mathbb{N})$
such that

\noindent (1) if $F,G\in\mathcal{P}_{f}(\mathcal{T})$ and $F\subsetneq G$,
then $ \max H(F) < \min H(G)$ and

\noindent (2) whenever $m\in\mathbb{N},G_{1},G_{2},\ldots,G_{m}\in\mathcal{P}_{f}(\mathcal{T}),G_{1}\subsetneq G_{2}\subsetneq\ldots\subsetneq G_{m}$
and for each $i\in\{1,2,\ldots,m\}$ , $f_{i}\in G_{i}$, one has
$\sum_{i=1}^{m}(\alpha(G_{i})+\sum_{t\in H(G_{i})}f_i(t))\in A$.
\end{defn}

In this article, the main object of study is partial semigroups. In \cite{AG18}, the author extended the Central Sets Theorem obtained by taking all possible sequences in commutative adequate partial semigroup. It states as

\begin{thm}\label{adeprod}
Let $(S,\ast)$ be a commutative adequate partial semigroup and let $C$
be a central subset of $S$. Let $\mathcal{T}_S$ be the set of all adequate sequences in $S.$ There exist functions $\alpha:\mathcal{P}_{f}(\mathcal{T}_S)\rightarrow S$
and $H:\mathcal{P}_{f}(\mathcal{T}_S)\rightarrow\mathcal{P}_{f}(\mathbb{N})$
such that
\begin{itemize}
\item[(1)] $F,G\in\mathcal{P}_{f}(\mathcal{T}_S)$ and $F\subsetneq G$,
then $\max H(F) < \min H(G)$ and
\item[(2)] whenever $m\in\mathbb{N}$, $G_{1},G_{2},\ldots,G_{m}\in\mathcal{P}_{f}(\mathcal{T}_S)$, 
$G_{1}\subsetneq G_{2}\subsetneq\ldots\subsetneq G_{m}$, and for each 
$i\in\{1,2,\ldots,m\}$, $f_i \in G_{i}$, one 
has\hfill\break 
$\prod_{i=1}^{m}(\alpha(G_{i})\ast\prod_{t\in H(G_{i})}f_i(t))\in C$.
\end{itemize}
\end{thm}
\begin{proof}
\cite[Theorem 2.4]{AG18}.
\end{proof} 
 In the same article, the author also characterized those sets in any adequate commutative partial semigroups which satisfy the new version of the Central Sets Theorem \cite[Theorem 2.4]{AG18} and introduced the analogous notion of $C$-set and $J$-set, namely respectively $C$-set and $J_{\delta}$-set in adequate partial semigroup as
 
\begin{defn}\label{cset}
Let $(S,\ast)$ be a commutative adequate partial semigroup and $A$ be a subset of $S.$  Let $\mathcal{T}_S$ be the set of all adequate sequences in $S.$\\
a) $A$ is said to be a $C$-set if and only if there exist functions
$\alpha:\mathcal{P}_{f}(\mathcal{T}_S)\rightarrow S$ and $H:\mathcal{P}_{f}(\mathcal{T}_S)\rightarrow\mathcal{P}_{f}(\mathbb{N})$
such that
\begin{itemize}
\item[(1)] if $F,G\in\mathcal{P}_{f}(\mathcal{T}_S)$ and $F\subsetneq G$,
then $\max H(F) < \min H(G)$ and
\item[(2)] whenever $m\in\mathbb{N}$, $G_{1},G_{2},\ldots,G_{m}\in\mathcal{P}_{f}(\mathcal{T}_S)$, 
$G_{1}\subsetneq G_{2}\subsetneq\ldots\subsetneq G_{m}$
and for each $i\in\{1,2,\ldots,m\}$, $f_i \in G_{i}$,
one has\hfill\break
$\prod_{i=1}^{m}(\alpha(G_{i})\ast\prod_{t\in H(G_{i})}f_i(t)\in A$.
\end{itemize}
b) $A$ is a $J_{\delta}$-set if and only if whenever 
$F\in\mathcal{P}_{f}(\mathcal{T}_S)$, $W\in\mathcal{P}_{f}(S)$, there exist
$a\in\sigma(W)$ and $H\in\mathcal{P}_{f}(\mathbb{N})$ such
that for each $f\in F$, $\prod_{t\in H}f(t)\in\sigma(W\ast a)$ and
$a\ast\prod_{t\in H}f(t)\in A$.\\
c) $J_{\delta}(S)=\{p\in\delta S:$ for all $A\in p\,,\,A$ is a $J_{\delta}$-set$\}$.
\end{defn}

 For a semigroup $S$, the $C$-sets, defined in purely combinatorial terms, are characterized as members of idempotents in $J(S)$. In \cite{AG18}, the author studied a similar kind of properties in the context of adequate partial semigroups and introduced the set $J_{\delta}(S)$. In this article, we first show that for a commutative adequate partial semigroup  sets are ideals in $\delta S$ ({\it cf.} Theorem 3.4). Next, we study the behaviour of these combinatorial sets under a surjective partial semigroup homomorphism ({\it cf.} Theorem 5.5., Theorem 5.6., Theorem 5.7., and Theorem 5.9.). With the help of these results we able to construct a way to construct enormous amount of $C$-sets which are not central.

\begin{notation}
$\mathcal{T}_S=$ The set of all adequate sequences in $S.$
\end{notation}

\section{Preliminaries}
In this section, we recall some definitions and results from partial semigroup context. For more details, the readers are referred to \cite{HS12}. We start with the following definition
\begin{defn}
(Partial semigroup) A partial semigroup is  defined as a pair $(G, \ast)$ where $\ast$ is an operation 
defined on a subset $X$ of $G \times G$ and satisfies the statement that for all 
$x$, $y$, $ z$ in $G$, $(x \ast y) \ast z = x \ast (y \ast z) $ in the sense that if either side 
is defined, so is the other and they are equal. A partial semigroup is commutative if ${x\ast y=y\ast x}$
for every ${(x,y)\in X}$.
\end{defn}
Partial semigroups which arise naturally our minds. 
\begin{exam}
Any semigroup is an obvious example of a partial semigroup.
\end{exam}

\begin{exam}
 Let us consider ${G=\mathcal{P}_f({\mathbb{N}})}= \{F \mid \emptyset \neq F \subseteq \mathbb{N}$ and $F$ is finite$\}$ and let ${X=\{(\alpha,\beta)\in G \times G \mid \alpha\cap\beta=\emptyset\}}$
be the family of all pairs of disjoint sets, and let ${\ast:X\rightarrow G}$
be the union. It is easy to check that this is a commutative partial
semigroup. We shall denote this partial semigroup as $(\mathcal{P}_f({\mathbb{N}}),\uplus)$.

\end{exam}

Next, we define homomorphisms between partial semigroups. 
\begin{defn}
Let $(S, \ast)$ and $(T, \ast^{\prime})$ be partial semigroups and let $h : S \rightarrow T.$ Then $h$ is a partial semigroup homomorphism if and only if whenever $y \in \phi_S(x),$ one has that $h (y) \in \phi_T(h(x))$ and $h(x \ast y)= h(x) \ast^{\prime} h(y).$ Here $\phi_S(s)$ stands for the set $\{t\in S \mid s\ast t$ is defined in $S\}$.
\end{defn}

For a semigroup, it is known that there is a notion of compactification, namely the Stone-$\check{C}$ech compactification, which is a compact right topological semigroup. One can do the same thing for any partial semigroup but that won't be a semigroup. To get a compact topological semigroup out of a partial semigroup we recall the following definitions.
\begin{defn}
Let $(S,\ast)$ be a partial semigroup.\\
(a) For $s\in S$, $\phi_S(s)=\{t\in S \mid s\ast t$ is defined in $S\}$.\\
(b) For $H\in\mathcal{P}_{f}(S)$, $\sigma_S(H)=\bigcap_{s \in H}\phi_S(s)$.\\
(c) $(S,\ast)$ is adequate if and only if $\sigma_S(H)\neq\emptyset$
for all $H\in\mathcal{P}_{f}(S)$.\\
(d) $\delta S=\bigcap_{x\in S}cl_{\beta S}(\phi_S(x))=\bigcap_{H\in\mathcal{P}_{f}(S)}cl_{\beta S}(\sigma_S(H))$.
\end{defn}

So, the partial semigroup $(\mathcal{P}_f({\mathbb{N}}),\uplus)$ is adequate. We are specifically interested in 
adequate partial semigroups as they lead to an interesting subsemigroup $\delta S$ of $\beta S$,
the Stone-\v{C}ech compactification of $S$ which  is itself
a compact right topological semigroup. Notice that adequacy of $S$
is exactly what is required to guarantee that $\delta S\neq\emptyset$.
If $S$ is, in fact, a semigroup, then $\delta S=\beta S$.

Now we recall some of the basic properties of the operation $\ast$
in $\delta S$.

\begin{defn}
Let $(S,\ast)$ be a partial semigroup. For $s\in S$ and $A\subseteq S$,
$s^{-1}A=\{t\in\phi_S(s) \mid s\ast t\in A\}$.
\end{defn}

\begin{lem}
Let $(S,\ast)$ be a partial semigroup, let $A\subseteq S$ and let
$a,b,c\in S$. Then $c\in b^{-1}(a^{-1}A)$ if and only if both $b\in\phi_S(a)$
and  $c\in(a\ast b)^{-1}A$.
In particular, if $b\in\phi_S(a)$, then $b^{-1}(a^{-1}A)=(a\ast b)^{-1}A$.
\end{lem}

\begin{proof}
\cite[Lemma 2.3]{HM01}
\end{proof}
\begin{defn}
Let $(S,\ast)$ be an adequate partial semigroup.\\
(a) For $a\in S$ and $q\in\overline{\phi_S(a)}$, $a\ast q=\{A\subseteq S \mid a^{-1}A\in q\}$.\\
(b) For $p\in\beta S$ and $q\in\delta S$, $p\ast q=\{A\subseteq S \mid \{a^{-1}A\in q\}\in p\}$.
\end{defn}

\begin{lem}
Let $(S,\ast)$ be an adequate partial semigroup.\\
(a) If $a\in S$ and $q\in\overline{\phi_S(a)}$, then $a\ast q\in\beta S$.\\
(b) If $p\in\beta S$ and $q\in\delta S$, then $p\ast q\in\beta S$.\\
(c) Let $p\in\beta S,q\in\delta S$, and $a\in S$. Then $\phi_S(a)\in p\ast q$
if and only if $\phi_S(a)\in p$.\\
(d) If $p,q\in\delta S$, then $p\ast q\in\delta S$.
\end{lem}

\begin{proof}
\cite[Lemma 2.7]{HM01}.
\end{proof}
\begin{lem}\label{1.12}
Let $(S,\ast)$ be an adequate partial semigroup and let $q\in\delta S$.
Then the function $\rho_{q}:\beta S\rightarrow\beta S$ defined by
$\rho_{q}(p)=p\ast q$ is continuous.
\end{lem}

\begin{proof}
\cite[Lemma 2.8]{HM01}.
\end{proof}
\begin{lem}
Let $p\in\beta S$ and let $q,r\in\delta S$. Then $p\ast(q\ast r)=(p\ast q)\ast r$.
\end{lem}

\begin{proof}
\cite[Lemma 2.9]{HM01}.
\end{proof}
\begin{defn}
Let $p=p \ast p \in \delta S$ and let $A \in p$. Then $A^{\ast} = \{x \in A \mid x^{-1}A \in p\}$.
\end{defn}

Given an idempotent $p \in \delta S$ and $A \in p$, it is immediate that $A^{\ast} \in p$.
\begin{lem}\label{1.15}
Let $p= p \ast p \in \delta S$, let $A \in p$, and let $x \in A^{\ast}$. Then $x^{-1}A^{\ast} \in p$.
\end{lem}

\begin{proof}
\cite[Lemma 2.12]{HM01}.
\end{proof}

As a consequence of the above results, we have that if $(S, \ast)$ is an adequate 
partial semigroup, then $(\delta S, \ast)$ is a compact right topological semigroup. 
Being a compact right topological semigroup, $\delta S$ contains
idempotents, left ideals, a smallest two-sided ideal, and minimal idempotents. 
Thus $\delta S$ provides a suitable environment for considering the notion of central sets and it defines as

\begin{defn} Let $(S,*)$ be an adequate partial semigroup. A set 
$C\subseteq S$ is {\it central\/} if and only if there is an idempotent
$p\in\overline C\cap K(\delta S)$.\end{defn}

In the Central sets Theorem for semigroup, we have studied that a necessary condition for the central sets along the line of sequences. Here, in the partial semigroup setting we can expect the similar thing for an adequate sequences instead of the normal one.

The notion of adequate sequence plays an important role in the study of central sets theorem  for the partial semigroups. So, we recall the definition as 

\begin{defn}
Let $(S,\ast)$ be an adequate partial semigroup and let\break
 $\left\langle y_{n}\right\rangle _{n=1}^{\infty}$
be a sequence in $S$. Then $\left\langle y_{n}\right\rangle _{n=1}^{\infty}$
is adequate if and only if $\prod_{n\in F}y_{n}$ is defined for each
$F\in\mathcal{P}_{f}(\mathbb{N})$ and for every $K\in\mathcal{P}_{f}(S)$,
there exists $m\in\mathbb{N}$ such that $FP(\left\langle y_{n}\right\rangle _{n=m}^{\infty})\subseteq\sigma_S(K)$.
\end{defn}

\begin{defn}
Let $W_1, W_2 \in \mathcal{P}_f(S)$, then define $W_1 \ast W_2=
\{w_1 \ast w_2 \mid w_1 \in W_1,  w_2 \in W_2 \; \text{and} \; w_1 \ast w_2 \;  \text{is defined}\}$.
\end{defn}

The sets which satisfy the new version of Central Sets Theorem \ref{adeprod} in adequate partial semigroup are said to be $C$-sets.

\section{Properties of $J_{\delta}$-sets}
This section concerns with the close look up on the $J_{\delta}$-sets and comes with some essential properties. We start with two essential lemmas.
\begin{lem} \label{prod ade}
Let $(S, \ast)$ be an adequate partial semigroup. Let $f$ be an adequate sequence in $S$ and let $\langle H_n \rangle_{n=1}^{\infty}$ be a sequence in $\mathcal{P}_f(\mathbb{N})$ such that $\max H_n < \min H_{n+1}$ for each $n \in \mathbb{N}.$ Define $g :\mathbb{N} \rightarrow S$ such that for each $n \in \mathbb{N}$, $g(n)= \prod_{t \in H_n} f(t).$ Then $g$ is an adequate sequence in $S.$
\end{lem}
\begin{proof}
To see that  $g$ is an adequate sequence, let $G \in \mathcal{P}_f(\mathbb{N})$ and let $H= \bigcup_{n \in G}H_n.$ Then $H \in \mathcal{P}_f(\mathbb{N}).$ Since $f$ is adequate then $\prod_{t \in H} f(t)$ is defined. Now $\prod_{n \in G}g(n)= \prod_{n \in G} \prod_{t \in H_n} f(t)= \prod_{t \in H} f(t),$ then $\prod_{n \in G} g(n)$ is defined. Let $K \in \mathcal{P}_f(S)$ be given. Then there exist $m \in \mathbb{N}$ such that $FP(\langle f(t) \rangle_{t=m}^{\infty}) \subseteq \sigma_S (K).$ Now we want to show that for $K \in \mathcal{P}_f(S),$ there exists $m \in \mathbb{N}$ such that $FP(\langle g(t) \rangle_{t=m}^{\infty}) \subseteq \sigma_S (K).$ Let $N \in \mathcal{P}_f(\{n \in \mathbb{N} : n \geqslant m\}).$ Then for each $n \in N$, $\min H_m \leqslant \min H_n.$ Let $H^{\prime}= \bigcup_{n \in N}H_n.$ Then $H^{\prime} \in \mathcal{P}_f(\{n \in \mathbb{N} : n \geqslant m\})$ and hence $\prod_{t \in H^{\prime}}f(t) \subseteq \sigma_S(K).$ Now $\prod_{n \in N}g(n)= \prod_{n \in N}(\prod_{t \in H_n}f(t))= \prod_{t \in H^{\prime}}f(t) \in \sigma_S(K).$ Therefore, $g$ is an adequate sequence.
\end{proof}

\begin{lem}\label{J prod}

Let $S$ be an adequate commutative partial semigroup, let $A$ be a $J_{\delta}$-set in $S.$ Let $F \in \mathcal{P}_f(\mathcal{T}_S),$ and let $\langle H_n \rangle_{n=1}^{\infty}$ be a sequence in $\mathcal{P}_f(\mathbb{N})$ such that for each $n$, $\max H_n < \min H_{n+1}.$ Then for all $W \in \mathcal{P}_f(S),$ there exist $a \in \sigma_S (W)$ and $G \in \mathcal{P}_f(\mathbb{N})$ such that for all $f \in F$, $\prod_{k \in G} \prod_{t \in H_k} \in \sigma_S(W \ast a)$ and $a \ast \prod_{k \in G} \prod_{t \in H_k} \in A. $

\end{lem}

\begin{proof}
For $f \in F,$ define $g_f : \mathbb{N} \rightarrow S$ by $g_f(k)=\prod_{t\in H_k}f(t).$ By Lemma \ref{prod ade}, $g_f$ is an adequate sequence in $S.$ Now since $A$ is a $J_{\delta}$-set, then  for $W \in \mathcal{P}_f(S)$ there exist $a \in \sigma_S(W)$ and $G \in \mathcal{P}_f(\mathbb{N})$ such that for each, $f \in F$, $\prod_{k \in G}g_f(k) \in \sigma_S(W \ast a)$ and $a \ast \prod_{k \in G}g_f(k) \in A.$ These imply $\prod_{k \in G} \prod_{t \in H_k} \in \sigma_S(W \ast a)$ and $a \ast \prod_{k \in G} \prod_{t \in H_k} \in A. $

\end{proof}

\begin{thm}\label{partition r}
Let $S$ be an adequate commutative partial semigroup. Let $A$ be a $J_{\delta}$-set in $S,$ and assume that $A=A_1 \cup A_2.$ Either $A_1$ is a $J_{\delta}$-set in $S$ or $A_2$ is a $J_{\delta}$-set in $S.$ 

\end{thm}

\begin{proof}
Suppose the conclusion of the statement of the theorem is false. That is both $A_1$ and $A_2$ are not $J_\delta$-sets. Then, we can pick $F_1, F_2 \in \mathcal{P}_f(\mathcal{T}_S)$ and $W_1, W_2 \in \mathcal{P}_f(S)$ such that for all $i \in \{1,2\}$ and for all $d \in \sigma_S(W_i)$ and $K \in \mathcal{P}_f(\mathbb{N})$ such that there exist $f \in F_i$ such that $d \ast \prod_{t\in K}f(t) \notin A_i \cap \sigma_S(W_i).$ Now, set $W = W_1 \cup W_2$ and $F = F_1 \cup F_2.$ Assume that $F = \{f_1, \cdots, f_p \}.$ 

Next, consider $B = \{w \mid w \; \text{is a word of length}\; n \; \text{over} \; \{ 1, \cdots,p\} \}.$ For each $w = b_1 \cdots b_n \in B$ define a function $g_w : \mathbb{N} \to S$ given by $$g_w(y) = \prod_{1 \leq i \leq n} f_{b_i}(ny +i)$$ where $y \in \mathbb{N}.$ Then, Lemma \ref{prod ade} guarantees that for each $w \in B$, $g_w$ is an adequate sequence. Define $G = \{g_w \mid w \in B \} \subseteq \mathcal{P}_f(\mathcal{T}_S).$ Since, $A$ is a $J_{\delta}$-set then for this $G$ and the set $W$ chosen earlier we get $a \in \sigma_S(W)$ and $H \in \mathcal{P}_f(\mathbb{N})$ such that for each $w \in B$ $$a \ast\prod_{y \in H} g_w(y) \in A \cap \sigma_S(W).$$ 

By \cite[Lemma 14.8.1]{HS12}, we can pick $n \in \mathbb{N}$ such that whenever the set $B$ is 2-colored there is a variable word $w(v)$ beginning and ending with a constant and without successive occurrences of $v$ such that $\{w(l): l \in \{ 1, \cdots, p\} \}$ is monochromatic.
 
 Define a $2$-coloring $\phi : B \to \{ 1,2\}$ on the set $B$ given by $$\phi(w) =\begin{cases} 1, & \text{if} \; a \ast\prod_{y \in H} g_w(y) \in A_1 \cap \sigma_S(W) \\ 2, & \text{otherwise} \end{cases}.$$ So, we can pick a variable word $w(v)$ beginning and ending with a constant and without successive occurrences of $v$ such that $\{w(l): l \in \{ 1, \cdots, p\} \}$ is monochromatic. Without loss of generality, we assume that $\phi(w(l)) =1.$ In other words, we assume that for each $l \in \{1, \cdots, p\}$, 
 \begin{equation}\label{par}
 a \ast \prod_{y \in H}g_{w(l)}(y) \in A_1 \cap \sigma_S(W).
 \end{equation}
 
 Suppose $w(v) = b_1\cdots b_n$ and $r$ is the total number of occurrence of $v$ in the word $w(v).$ Then we can write down the set $\{j \in \{1, \cdots, n\} \mid b_j \in \{ 1, \cdots, p\} \} = \bigcup_{i =1}^{r+1}L(i)$ and $\{j \in \{1, \cdots, n\}\mid b_j = v\} = \{s(1), \cdots, s(r) \}$ for some function $s : \{ 1, \cdots, r\} \to \mathbb{N}$ such that $max L(x) < s(x) < min L(x+1)$. Precisely, the function $s$ refers the position function for $v$ in $w(v).$ The construction of $L(i)$'s can be understand by the following example. Let $w(v) = 2v31v12.$ Then $r =2$, $L(1) =\{1\},$ $L(2) =\{ 3,4\},$ and $L(3) = \{ 6,7\}.$ So, $L(i)$ is the set of $j \in \{  1, \cdots, n\}$ such $b_j \in \{1, \cdots, p \}$ and $b_j$'s are in between $(i-1)^{th}$ and $i^{th}$ occurrences of $v$ in $w(v).$
 Next, we try to understand the term $a \ast \prod_{y \in H}g_{w(l)}(y).$ Using the previous discussion so far we have
 \begin{align*}
 a \ast \prod_{y \in H}g_{w(l)}(y)& = a \ast \prod_{y \in H} \prod_{1 \leq i \leq n}f_{b_i}(ny +i) \\ & =a \ast \prod_{y \in H}\prod_{k=1}^{r+1} \prod_{i \in L(k)} f_{b_i}(ny+i) \ast \prod_{y \in H} \prod_{i=1}^{r} f_l(ny + s(i)) \\ & =d \ast \prod_{t \in K}f_l(t) 
 \end{align*}
 
 Where $$d = a \ast \prod_{y \in H}\prod_{k=1}^{r+1} \prod_{i \in L(k)} f_{b_i}(ny+i)\; \text{and}\; K = \bigcup_{i=1}^r \{ny + s(i)\mid y \in H \}$$. Then definitely $K$ is a finite subset of $\mathbb{N}.$ Also, \eqref{par} tells us $W \ast a \ast \prod_{y \in H}g_{w(l)}(y)$ is defined i.e. $W \ast d \ast \prod_{t \in K}f_l(t)$ is defined. Using associativity, this implies $W \ast d$ is defined. Hence $d \in \sigma_S(W)$ and $d \ast \prod_{t\in K}f_l(t) \in A_1 \cap \sigma_S(W) \subset A_1 \cap \sigma_S(W_1).$ This a contradiction to our assumption.
 \end{proof}
 \begin{thm}\label{ideal}
Let $(S, \ast)$ be an adequate commutative partial semigroup. Then $J_{\delta }(S) $ is a closed two sided ideal of $\delta S.$
\end{thm}

\begin{proof}
Let $A$ be a $J_{\delta}$-set in $S.$ Then Theorem  \ref{partition r} yields the  $J_{\delta}$-sets are partition regular. Next, we claim that $J_{\delta}(S) \neq \emptyset$. To prove the claim, we consider the sets $\mathfrak{R} = \{ A \subseteq S \mid A \; \text{a is } \; J_{\delta}\text{-set}\}$ and  $\mathfrak{A} = \{ \phi_S(s) \mid s \in S\}.$ Observe that for $\sigma_S(F) = \cap_{s \in F}\phi_S(s)$ we always can find a $J_{\delta}$-set $A$ such that $A \subseteq \sigma_S(F).$ This gives that for any finite intersections of the members of $\mathfrak{A}$ lies inside $\mathfrak{R}^{\uparrow}.$ Therefore, using Theorem \cite[Theorem 3.11]{HS12} we get a ultrafilter $p$ of $S$ such that $\mathfrak{A} \subseteq p \subseteq \mathfrak{R}^{\uparrow}.$ The first inclusion implies that  for each $s \in S,$ $\phi_S(s) \in p$, therefore, $p \in \delta S.$ This forces $p \in J_{\delta}(S).$ Hence, the claim is established.  Next, we want to show that $J_{\delta}(S)$ is closed. Let $p \in \delta S \setminus J_{\delta}(S),$ then there exists $B \in p$ such that $B$ is not a $J_{\delta}$-set, therefore, $\overline{B} \cap J_{\delta}(S) = \emptyset ,$ i.e., $J_{\delta}(S) $ is closed.

We want to prove that $J_{\delta }(S)$ is a two sided ideal of $\delta S,$ i.e., for all $p \in \delta S$, $q \in J_{\delta }(S)$ imply that $ p \ast q \in J_{\delta}(S)$ and $q \ast p \in J_{\delta}(S).$

Let $A \in p \ast q.$ We want to show that $A$ is a $J_{\delta }$-set.

Since $A \in p \ast q,$ then $\{a \in S \mid a^{-1}A \in q\} \in p. $ Pick $a \in S$ such that $a^{-1}A \in q.$ Now $q \in J_{\delta }(S),$ then $a^{-1}A$ is a $J_{\delta}$-set. Therefore, for $F \in \mathcal{P}_f(\mathcal{T}_S)$ and $W \in \mathcal{P}_f(S),$ there exist $b \in \sigma_S(W \ast a)$ and $H \in \mathcal{P}_f(\mathbb{N})$ such that $\prod_{t \in H}f(t) \in \sigma_S (W \ast a \ast b)$ and $b \ast \prod_{t \in H}f(t) \in a^{-1}A= \{d \in \phi_S(a) \mid a \ast d \in A\}.$ This implies $a \ast b \ast \prod_{t \in H}f(t) \in A.$ Now define $c=a \ast b.$ Then for $F \in \mathcal{P}_f(\mathcal{T}_S)$ and $W \in \mathcal{P}_f(S),$ there exist $c \in \sigma_S(W)$ and $H \in \mathcal{P}_f(\mathbb{N})$ such that for each $f \in F$, $\prod_{t \in H}f(t) \in \sigma_S (W \ast c)$ and $c \ast \prod_{t \in H}f(t) \in A.$ Therefore, $A \in p \ast q.$

Now let $A \in q \ast p.$ We want to show that $A$ is a $J_{\delta}$-set.

Let $B=\{a \in S \mid a^{-1}A \in p \}.$ Since $A \in q \ast p,$ then $\{a \in S \mid a^{-1}A \in p \} \in q.$ This implies $B \in q.$ Now $q \in J_{\delta}(S),$ then $B$ is a $J_{\delta}$-set. Let $F \in \mathcal{P}_f(\mathcal{T}_S)$ and $F=\{f_1, f_2, \cdots, f_k\}$. Let $W_{0} \in \mathcal{P}_f(S),$ then there exist $b \in \sigma_S(W_{0})$  and $H \in \mathcal{P}_f(\mathbb{N})$ such that for each $f \in F$, $\prod_{t \in H}f(t) \in \sigma_S(W_{0} \ast b) $ and $b \ast \prod_{t \in H}f(t) \in B.$ Therefore, $(b \ast \prod_{t \in H}f(t))^{-1}A \in p$ for all $f \in F$ and this yields $\bigcap\limits_{f \in F}(b \ast \prod_{t \in H}f(t))^{-1}A \in p.$ Now since $p \in \delta S,$ then $\sigma_S(W) \in p$ for all $W \in \mathcal{P}_f(S).$ Therefore, $\sigma_S(W) \cap \bigcap\limits_{f \in F}(b \ast \prod_{t \in H}f(t))^{-1}A \in p$ for all $W \in \mathcal{P}_f(S).$ So it must be nonempty. Take $W= W_{0} \ast F_0,$ where $F_0=\{b \ast \prod_{t \in H}f_1(t),b \ast \prod_{t \in H}f_2(t), \cdots,b \ast \prod_{t \in H}f_k(t)\} $ and pick  $$c \in \sigma_S(W_0 \ast F_0) \cap \bigcap\limits_{f \in F}(b \ast \prod_{t \in H}f(t))^{-1}A.$$  Now choose $a=b \ast c.$ Then for each $f \in F$, $\prod_{t \in H}f(t) \in \sigma_S(W_0 \ast a)$ and $a \ast \prod_{t \in H}f(t) \in A.$ This suffices $A$ is a $J_{\delta}$-set.

\end{proof}

\begin{cor}
Let $(S, \ast )$ be an adequate commutative partial semigroup. Let $A \subseteq S.$ If $A$ is a central set in $S,$ then $A$ is a $C$-set in $S.$

\end{cor}
\begin{proof}
Since $A$ is a central set in $S,$ then there is an idempotent $p$ such that $p \in K(\delta S) \cap \bar{A}.$ By lemma \ref{ideal}, $J_{\delta}(S)$ is a two sided ideal of $\delta S. $ Then $K(\delta S) \subseteq J_{\delta }(S).$  Therefore, by \cite[Theorem 3.4]{AG18}, $A$ is a $C$-set in $S.$
\end{proof}

\section{Construction of $J_{\delta}$-sets}
This section deals with some technical construction of $J_{\delta}$-sets. We start with some definitions.

Let $\omega$ denotes the first infinite ordinal and each ordinal is the set of its predecessors. In particular, $[0]=\emptyset$ and for $n \in \mathbb{N}$, $[n]=\{0,1,\cdots, n-1\}.$
\begin{defn}
If $f$ is a function and domain$(f)=[n] \in \omega,$ then for all $x \in S$, $f^{\frown}x = f \cup \{ (n, x)\}$. Precisely, it means we extend the domain of $f$ to $[n+1]$ by defining $f(n)=x.$

\end{defn}

\begin{defn}
Let $T=\{f \mid f : [n] \rightarrow S \},$ i.e., $T$ is a set of functions whose domains are members of $\omega.$ For each $f \in T,$ define $B_f(T)=\{ x \mid f^{\frown}x \in T\}.$

Using these notions in hand we can stated the following
\end{defn}
\begin{lem}\label{words}
Let $S$ be an adequate partial semigroup and let $p \in \delta S.$ Then $p$ is an idempotent if and only if for each $A \in p$ there is a nonempty set $T$ of functions such that 
\begin{itemize}
\item[(a)] For all $f \in T,$ domain$(f)\in \omega$ and range$(f) \subseteq A.$
\item[(b)] For all $f \in T$, $B_f(T) \in p.$
\item[(c)] For all $f \in T$ and all $x \in B_f(T)$, $B_{f^{\frown}x}(T) \subseteq x^{-1}B_f(T).$
\end{itemize}

\end{lem}
\begin{proof}
We claim that $p$ is an idempotent element in $\delta S,$ i.e., $\{x \in S \mid x^{-1}A \in p\} \in p.$ Now there is given $A \in p$ and $T$ which satisfies the above conditions. Pick $f \in T.$ Then $B_f(T) \in p.$ Therefore, if we will prove that $B_f(T) \subseteq  \{x \in S \mid x^{-1}A \in p\},$ then our claim will be proved. Let $x \in B_f(T),$ then $f^{\frown}x \in T. $ Therefore, by (b) $B_{f^{\frown}x}(T) \in p.$ Now by (c), $B_{f^{\frown}x}(T) \subseteq x^{-1}B_f(T) \subseteq x^{-1}A.$ Thus $x^{-1}A \in p,$ and so $p$ is an idempotent element.

Conversely, let $p$ be an idempotent element in  $\delta S$ and let $A \in p.$ For any $B \in p, $ let define $B^*=\{x \in B \mid x^{-1}B \in p\}.$ Then $B^* \in p$ and by \cite[Lemma 2.12]{HM01}, for $x \in B^*$, $x^{-1}B^* \in p.$ To prove the result we inductively construct a filtration of $T = \bigcup\limits_{[n] \in \omega} T_n$ where $T_n=\{f \in T \mid \text{domain}(f)=[n]\}$ and for each $f\in T_n,$ define $B_f=B_f(T).$ Note that $T_{\emptyset} = \{ 0\}$. Now set $B_{\emptyset} = A^*.$ 

Now let $[n] \in \omega$  and assume that we have defined $T_k$ for $k \leq n$ and defined $B_f$ for $f \in T_k$ such that 

\begin{itemize}
\item[(a)] $T_k $ is a set of functions with domain $[k] $ and range contained in $A.$
\item[(b)] If $f \in T_k$ and $x \in B_f,$ then $B_f \in p$ and $x^{-1}B_f \in p.$
\item[(c)] If $k < n$, $f \in T_k,$ and $x \in B_f,$ then $B_{f^{\frown}x}=(x^{-1}B_f)^*.$
 
\end{itemize}
If $n=0$, $T_0=\{\emptyset\}.$ Then (a) is trivially true.
For (b), $B_{\emptyset}=A^* \in p$ and if $x \in A^*,$ then $x^{-1}A^* \in p.$ Again (c) is vacuously true. 

Now, we will check the hypotheses for $T_{n+1}.$

$T_{n+1}=\{f \in T \mid \text{domain}(f)=[n+1] \}.$ So, for any $f \in T_{n+1}$ can be written as $f = {f|_{[n]}}^{\frown}x$ where $f(n)=x.$ So, we can write $T_{n+1}=\{f^{\frown} x \mid f \in T_n \; \text{and} \; x \in B_f\}.$ Now let $g \in T_{n+1},$ then we can take $g=f^{\frown}x$ for $f \in T_n$ and $x=g(n).$ Now by (b), for $f \in T_n$ and $x \in B_f$, $x^{-1}B_f \in p.$ Let $B_g=(x^{-1}B_f)^*.$ Now let $y \in B_g= (x^{-1}B_f)^*,$ then $y^{-1}B_g \in p.$

Therefore, the hypotheses are true for $T_{n+1}.$
\end{proof}

Now we are in a position to construct a huge amount of $J_{\delta}$-sets using the set $B_f(T)$ and some particular idempotent elements in the algebra $\delta S.$ Precisely, we have

\begin{thm}
Let $(S, \ast)$ be an adequate partial semigroup and let $A \subseteq S.$ If there is an idempotent $p \in \bar{A} \cap J_{\delta }(S).$ Then there is a non-empty set $T$ of functions such that:
\begin{itemize}
\item[(i)]  For all $f \in T$, domain$(f)\in \omega$ and range$(f) \subseteq A.$
\item[(ii)] For all $f \in T$ and all $x \in B_f(T)$, $B_{f^{\frown}x}(T) \subseteq x^{-1}B_f(T).$
\item[(iii)] For all $F \in \mathcal{P}_f(T)$, $\bigcap_{f \in F}B_f(T)$ is a $J_{\delta }$-set.
\end{itemize}

Moreover, there is a downward directed family $\langle A_F\rangle_{F \in I}$ of subsets of $A$ such that:
\begin{itemize}
\item[a)] For all $F \in I$ and all $x \in A_F,$ there exists $G \in I$ such that $A_G \subseteq x^{-1}A_F.$
\item[b)] For each $\mathcal{F} \in \mathcal{P}_f(I)$,
 $\bigcap_{F \in \mathcal{F}}A_F$ is a $J_{\delta}$-set.
 \end{itemize}
\end{thm}

\begin{proof}

 Since $p$ is an idempotent element in $\delta S$ and $A \in p,$ then by Lemma \ref{words}, (i) and  (ii) are true and for each $f \in F$, $B_f(T) \in p.$ Therefore, $\bigcap_{f \in F}B_f(T) \in p.$ Now since $p \in J_{\delta}(S),$ then $\bigcap_{f \in F}B_f(T)$ is a $J_{\delta}$-set.  This concludes part (iii).

For the existence of a downward directed family we do the following construction.  Let $I \in \mathcal{P}_f(T).$ For $F \in I,$ define $A_F=\bigcap_{f \in F}B_f(T).$  For part a),  consider $G=\{f^{\frown}x \mid f \in F\}.$ Now by Lemma \ref{words}, for each $f \in F$, $B_{f^{\frown}x}(T) \subseteq x^{-1}B_f(T).$ This implies that $A_G \subseteq x^{-1}A_F.$  Next, for part b), using the previous arguments, we obtain  $A_F$ is a $J_{\delta}$-set. Given $\mathcal{F} \in \mathcal{P}_f(I),$ if $H= \bigcup \mathcal{F},$ then $\bigcap_{F \in \mathcal{F}}A_F=A_H. $ Therefore, $A_H$ is a $J_{\delta }$-set.
\end{proof}
and also, we can readily reduce
\begin{cor}
Let $(S, \ast)$ be a countable adequate partial semigroup and let $A \subseteq S.$ If there is an idempotent $p \in \bar{A} \cap J_{\delta}(S),$ then there is a decreasing sequence $\langle A_n \rangle_{n=1}^{\infty}$ of subsets of $A$ such that:
\begin{itemize}
\item[(i)] For all $n \in \mathbb{N}$ and all $x \in A_n,$ there exists $m \in \mathbb{N}$ such that $A_m \subseteq x^{-1}A_n.$
\item[(ii)] For all $n \in \mathbb{N}$, $A_n$ is a $J_{\delta}$-set.
\end{itemize}
\end{cor}
\begin{proof}
Let $S$ be countable. Then $T$ is also countable. So identify $T$ as $\{f_n \mid  n \in \mathbb{N}\}$. For $n \in \mathbb{N},$ let $A_n=\bigcap_{k=1}^n B_{f_k}(T).$ Then each $A_n$ is a $J_{\delta}$-set. Let $x \in A_n.$ Now pick $m \in \mathbb{N}$ such that $\{{f_k}^{\frown}x \mid k \in \{1,2, \cdots, n\}\} \subseteq \{f_1,f_2, \cdots, f_m\}.$ Then $A_m \subseteq x^{-1}A_n.$
\end{proof}

\section{Construction of C-sets which are not central}
It is known that the combinatorial sets such as central sets, $C$-sets,  and $J_{\delta}$-sets have an extra importance in the study of combinatorics. Also, it is elementary from the definition that every central sets are $C$-sets. In this section, we try to find out a way to construct  a $C$-set which are not central. Before go into that we start with few basic facts.

\begin{lem}\label{semigroup}
Let $A$ and $B$ be semigroups and let $f: A \rightarrow B$ be a surjective homomorphism. If $A$ has a smallest ideal, show that $B$ does as well and that $K(B)=f(K(A)).$
\end{lem}
\begin{proof}
Let $b \in B$ and $y \in f(K(A)).$ Since $f $ is surjective, there exists $a \in A$ such that $f(a)=b.$ Now since $y \in f(K(A)),$ then there exists $x \in K(A)$ such that $f(x)=y.$ As $a \cdot x \in K(A),$ then $b \cdot y= f(a) \cdot f(x)=f(a \cdot x) \in f(K(A)).$ Therefore, $f(K(A))$ is a left ideal in $B.$ Similarly, we can show that it is a right ideal in $B.$

$K(A)=\bigcup\{L \mid L  \; \text{is a minimal left ideal of} \; A\}.$ Then, it follows that  $$f(K(A))= \bigcup\{f(L)\mid L  \;  \text{is a minimal left ideal of} \; A\}.$$  Let $J \subseteq f(L),$ where $J$ is an ideal of $B.$ Then $f^{-1}(J) \subseteq L.$ But $L$ is minimal left ideal of $A,$ then $J=f(L).$ Therefore, $f(L)$ is minimal left ideal of $B.$ Then $f(K(A)) \subset K(B).$ But $K(B)$ is the smallest ideal, then $f(K(A))=K(B).$
\end{proof}

\begin{cor}\label{6.0.2}
Let $f : A \rightarrow B$ be a semigroup homomorphism. If $A$ has a smallest ideal, then $f(K(A))=K(f(A)).$

\end{cor}
\begin{proof}
If $f : A \rightarrow B$ is a semigroup homomorphism, then $f : A \rightarrow f(A)$ is a surjective homomorphism. Then by Lemma \ref{semigroup},  if $A$ has smallest ideal, then $f(A)$ also has smallest ideal and $f(K(A))= K(f(A)).$ 
\end{proof}

\begin{lem}\label{semihomo}
Let $S$ and $T$ be adequate partial semigroups. Let $h : S \rightarrow T$ be a surjective partial semigroup homomorphism, i.e., $h(S)=T.$ Let $\tilde{h} : \beta S \rightarrow \beta T$ be the continuous extension of $h.$ Then $\tilde{h}(\delta S) \subseteq \delta T$ and $\tilde{h} \mid_{\delta S}$ is a semigroup homomorphism.
\end{lem}

\begin{proof}
For the proof, see \cite[Theorem 4.22.3]{HS12}.
\end{proof}

\begin{rmk}\label{rmk}
By Lemma \ref{semihomo} and Corollary \ref{6.0.2}, $\tilde{h}(K(\delta S))= K(\tilde{h}(\delta S)).$
\end{rmk}

\begin{thm} \label{Jset}
Let $S$ and $T$ be two adequate commutative partial semigroups, let $h: S\rightarrow T$ be a surjective partial semigroup homomorphism.
 If $A$ is a $J_{\delta}$-set in $S,$ then $h(A)$ is a $J_{\delta}$-set in $T.$

\end{thm}

\begin{proof}

 Given $A$ is a $J_{\delta}$-set in $S.$ Let $F \in \mathcal{P}_f(\mathcal{T}_T)$ and $W \in \mathcal{P}_f(T).$ Pick $k: T \rightarrow S$ such that $h\circ k(x)=x$ locally (pointwise). Construct $G=k(F)=\{k\circ f \mid  f \in F\}$ and $W_S= k(W).$ Then $G \in \mathcal{P}_f(\mathcal{T}_S)$ and $W_S \in \mathcal{P}_f(S).$ Since $A$ is a $J_{\delta}$-set in $S,$ then there exist $a \in \sigma_S (W_S)$ and $H \in \mathcal{P}_f(\mathbb{N})$ such that for each $g \in G$,

\begin{equation}\label{2}
\prod_{t \in H}g(t) \in \sigma_S(W_S \ast a) \; \mbox{and} \; a \ast \prod_{t \in H}g(t) \in A.
\end{equation}  Choose $b \in T$ such that $b=h(a)$ and $$h(a) \in h  (\sigma_S(W_S)) \subseteq \sigma_T(h(W_S)) = \sigma_T(h \circ k(W))= \sigma_T(W)$$ Now from one part of the condition (6.1), we have $$h(\prod_{t \in H}g(t) )\in h(\sigma_S(W_S \ast a))$$ which implies $\prod_{t \in H}h(g(t)) \in \sigma_T(h(W_S) \ast h(a)),$ and, so that $\prod_{t \in H}h\circ k \circ f(t) \in \sigma_T(W \ast b).$ This implies $\prod_{t \in H}f(t) \in \sigma_T(W \ast b)$ and the other part of the condition \eqref{2} gives us that  $h(a \ast \prod_{t \in H}g(t)) \in h(A).$ Since $h$ is partial semigroup homomorphism, then we get, $h(a) \ast \prod_{t \in H}h(g(t)) \in h(A)$ which clearly implies that $b \ast \prod_{t \in H}f(t) \in h(A).$ Therefore, $h(A)$ is a $J_{\delta}$-set in $T.$ 
\end{proof}

\begin{thm}\label{central}
Let $S$ and $T$ be two adequate commutative partial semigroups, let $h: S\rightarrow T$ be a surjective partial semigroup homomorphism. Let $\tilde{h} : \beta S \rightarrow \beta T$ be the continuous extension of $h.$ If $\tilde{h}: \delta S \rightarrow \delta T$ induces surjective map then
\begin{itemize}
\item[(a)] $\tilde{h}(K(\delta S))= K(\delta T).$ 

\item[(b)] If $A$ is a central set in $S,$ then $h(A)$ is a central set in $T.$ 
\item[(c)] If $A$ is a central set in $T,$ then $h^{-1}(A)$ is a central set in $S.$ 
\end{itemize}
 
\end{thm} 
\begin{proof}
(a) Given that  $\tilde{h}: \delta S \rightarrow \delta T$ is surjective semigroup homomorphism. Being compact right topological semigroup $\delta S$ has a smallest two sided ideal, then by Lemma \ref{semigroup},  $\delta T$ has also smallest two sided ideal and $K(\delta T)= \tilde{h}(K(\delta S)).$

(b) Let $A$ be a central set in $S,$ then there is an idempotent $p$ such that $p \in K(\delta S) \cap \overline{A}.$ Since $\tilde{h} \mid_{\delta S}$ is a semigroup homomorphism by Lemma \ref{semihomo} and hence $\tilde{h}(p)$ is an idempotent element and by part (a), $\tilde{h}(K(\delta S))= K(\delta T).$ Therefore, $\tilde{h}(p)$ is contained in $K(\delta T),$ and, so $\tilde{h}(p)$ is an idempotent element in $\overline{h(A)} \cap K(\delta T).$

(c) Pick an idempotent element $p$ in $K(\delta T) \cap \overline{A}.$ By part (a), pick $q \in K(\delta S)$ such that $\tilde{h}(q)=p.$ Now pick a minimal left ideal $L$ of $ \delta S$ such that $q \in L.$ Then $L \cap \tilde{h}^{-1}(\{p\})$ is a compact subsemigroup of $\delta S.$ So there is an idempotent $r \in L \cap \tilde{h}^{-1}(\{p\}).$ Since $A \in p$ and $\tilde{h}(r)=p$, then $h^{-1}(A) \in r.$ Thus, $h^{-1}(A)$ is a central set in $S.$
\end{proof}

\begin{thm}\label{cset}
Let $S$ and $T$ be two adequate commutative partial semigroups, let $h: S\rightarrow T$ be a surjective partial semigroup homomorphism. Let $\tilde{h} : \beta S \rightarrow \beta T$ be the continuous extension of $h.$

\begin{itemize}
\item[(a)] $\tilde{h}(J_{\delta}(S)) \subseteq J_{\delta }(T)$
\item[(b)] If there is an idempotent $p \in \overline{A} \cap J_{\delta}(S),$ then $h(A)$ is a $C$-set in $T.$

\item[(c)]If $p \in \overline{A} \cap \tilde{h}(J_{\delta }(S)),$ then $h^{-1}(A)$ is a $C$-set in $S .$
\end{itemize}
\end{thm}

\begin{proof}
(a) Let $p \in J_{\delta }(S).$ We need to show that $\tilde{h}(p) \in J_{\delta}(T).$ Now let $A \in \tilde{h}(p).$ Then $h^{-1}(A) \in p.$ Therefore, $h^{-1}(A)$ is a $J_{\delta}$-set in $S.$ Now by Lemma \ref{Jset}, $A=h(h^{-1}(A))$ is a $J_{\delta}$-set in $T.$ Therefore, $\tilde{h}(p) \in J_{\delta}(T).$

(b) Let $p \in \overline{A} \cap J_{\delta}(S).$ Since $\tilde{h} \mid_{\delta S}$ is a semigroup homomorphism by Lemma \ref{semihomo} and hence $\tilde{h}(p)$ is an idempotent element in $\delta T.$ Now since $p \in \overline{A},$ then $h(A) \in \tilde{h}(p),$ and  since $p \in J_{\delta}(S),$ then $p \in \delta S $ such that for all $B \in p$, B is a $J_{\delta}$-set in $S.$ Therefore, $\tilde{h}(p) \in \delta T$ where $h(B) \in \tilde{h}(p)$ and $h(B)$ is a $J_{\delta}$-set in $T$ by Lemma \ref{Jset}. 

Now let $D\in \tilde{h}(p).$ We want to show that $D$ is a $J_{\delta}$-set in $T.$ Now for a surjective map $h$, $D=h(h^{-1}(D)).$ Since $h^{-1}(D) \in p,$ then $h^{-1}(D)$ is a $J_{\delta}$-set in $S$ and again by Lemma \ref{Jset} $D$ is a $J_{\delta}$-set in $T.$  Therefore, we get an idempotent element in $\overline{h(A)} \cap J_{\delta}(T).$ Then by  Theorem \cite[Theorem 3.4]{AG18}, $h(A)$ is a $C$-set in $T.$


 (c) Let $p$ be an idempotent element such that $p \in \overline{A} \cap \tilde{h}(J_{\delta }(S)).$ By Theorem \ref{ideal}, $J_{\delta}(S)$ is closed two sided ideal of $\delta S.$ Now being a closed subset of compact set $\delta S$, $J_{\delta }(S)$ is compact and $J_{\delta}(S) \cap \tilde{h}^{-1}(\{p\})$ is a compact subsemigroup of $\delta S,$ where $\tilde{h}^{-1}(\{p\})=\{r \in \delta S \mid \tilde{h}(r)=p\}.$ Therefore, there is an idempotent $q \in J_{\delta}(S) \cap \tilde{h}^{-1}(\{p\}).$ Since  $A \in p$ and $\tilde{h}(q)=p,$ then $h^{-1}(A) \in q.$ Then by  Theorem \cite[Theorem 3.4]{AG18}, $h^{-1}(A)$ is a $C$-set in $S.$
 \end{proof}
 
 The following corollary is a direct consequence of Theorem \ref{central} and Theorem \ref{cset}.
\begin{cor}\label{cor}
Let $S$ and $T$ be two adequate commutative partial semigroups, let $h: S\rightarrow T$ be a surjective partial semigroup homomorphism. Let $\tilde{h} : \beta S \rightarrow \beta T$ be the continuous extension of $h$ and  $\tilde{h}: \delta S \rightarrow \delta T$ induces surjective map. If there is an idempotent $p \in \overline{A} \cap \tilde{h}(J_{\delta }(S)),$ where $A$ is not a central set in $T,$ then  $h^{-1}(A)$ is a $C$-set in $S$ but not a central set in $S.$
\end{cor}

\begin{exam}
Define $h:(\mathcal{P}_f(\mathbb{N}), \uplus) \to (\mathbb{N},+) $ by $h(A)=|A|,$ where $|A|$ is cardinality of the set $A.$ Then, $h$ is a surjective partial semigroup homomorphism. Let $\tilde{h} : \beta \mathcal{P}_f(\mathbb{N}) \rightarrow \beta \mathbb{N}$ be the continuous extension of $h.$ Then, $\tilde{h}(\delta \mathcal{P}_f(\mathbb{N}))= \beta \mathbb{N}.$ In \cite [Theorem 2.8]{HS9} Hindman and Strauss gave an example which is a $C$-set in $(\mathbb{N}, +),$ but not central. Since the set $A$ produced in \cite[Theorem 2.8]{HS9} is a $C$-set, then there is an idempotent element $p \in \overline{A} \cap J(\mathbb{N},+)$ by \cite[Theorem 2.5]{HS09}. As $J_{\delta}(\mathbb{N}, +) =J(\mathbb{N}, +),$ therefore, we have that $p \in \overline{A} \cap J_{\delta}(\mathbb{N}, +).$ So, in our case, $\; \tilde{h}(J_{\delta}(\mathcal{P}_f(\mathbb{N}), \uplus))= J_{\delta}(\mathbb{N}, +)$. Thus, Theorem \ref{cset} yields $h^{-1}(A)$ is a $C$-set, but by Corollary \ref{cor}, it is not central set in $(\mathcal{P}_f(\mathbb{N}), \uplus).$
\end{exam}

\end{document}